\documentclass{amsart}

\usepackage{graphicx}
\usepackage{tikz}

\newtheorem{theorem}{Theorem}[section]
\newtheorem{lemma}[theorem]{Lemma}
\newtheorem{prop}[theorem]{Proposition}

\theoremstyle{definition}
\newtheorem{definition}[theorem]{Definition}

\theoremstyle{remark}
\newtheorem{remark}[theorem]{Remark}

\numberwithin{equation}{section}

\newcommand{\R}{\mathbb{R}}

\newcommand{\N}{\mathbb{N}}
\newcommand{\Z}{\mathbb{Z}}

\newcommand{\vers}{\rightarrow}

\begin{document}

\title{COUNTEREXAMPLES TO RUELLE'S INEQUALITY IN THE NONCOMPACT CASE}

\author{Felipe Riquelme}
\address{IRMAR-UMR 6625 CNRS,
Universit\'e de Rennes 1, Rennes 35042}
\email{felipe.riquelme@univ-rennes1.fr}
\thanks{The author wish to thank Pierre Arnoux, Vincent Delecroix, S\'ebastien Gou\"ezel, Godofredo Iommi, Fran\c{c}ois Ledrappier, Barbara Schapira and Juan Souto for many useful discussions and commentaries that improved the paper.}


\date{\today}



\begin{abstract} In this paper we show that there exists a large family of smooth dynamical systems defined over noncompact spaces that does not satisfy Ruelle's inequality between entropy and Lyapunov exponents.
\end{abstract}

\maketitle

\section{Introduction}

Let $f:M\vers M$ be a $C^{1}$-diffeomorphism of a Riemannian manifold and let $\mu$ be an $f$-invariant probability measure on $M$. The measure-theoretic entropy of $f$ with respect to $\mu$, denoted by $h_{\mu}(f)$, is an ergodic invariant measuring the exponential growth rate of the complexity of the dynamics from the point of view of $\mu$. The set of (Lyapunov-Perron) regular points in $M$, denoted by $\Lambda$, is the set of points where the asymptotic eigenvalues of the linearized dynamics of $df$ are (somehow) well defined. These eigenvalues are called Lyapunov exponents of $f$ (see definition in section 2.2). When $M$ is compact the set $\Lambda$ has full measure (see \cite{MR0240280}). Moreover, Ruelle's inequality (see \cite{MR516310}) says that the measure-theoretic entropy of $f$ with respect to $\mu$, denoted by $h_{\mu}(f)$, satisfies
$$h_{\mu}(f)\leq \int \chi^{+}(x)d\mu(x),$$
where $\chi^{+}(x)$ is the sum of the positive Lyapunov exponents at $x\in \Lambda$.\\

When $M$ is a noncompact Riemannian manifold, no general such statement is known. In \cite{MR872698} the authors proved that Ruelle's inequality holds for diffeomorphisms with singularities on compact manifolds under some technical assumptions. Note that these diffeomorphisms can be understood like diffeomorphisms defined on noncompact manifolds. In any case, the compacity of the underlying manifold seems to be a crucial hypothesis.\\

The aim of this paper is to show that there exist diffeomorphisms on noncompact manifolds for which Ruelle's inequality is no longer satisfied. More precisely, we have
\newpage
\begin{theorem}\label{T1} For all $h\in (0,\infty]$ there exists a noncompact Riemannian manifold $M$, a $C^{\infty}$-diffeomorphism $f:M\vers M$ and a $f$-invariant probability measure $\mu$ over $M$, whose measure-theoretic entropy satisfies $h_{\mu}(f)=h$ and such that $\mu$-almost everywhere, the Lyapunov exponents are equal to zero. In other words, we have
$$0=\int \chi^{+} d\mu < h_{\mu}(f)\leq \infty.$$
\end{theorem}

The key idea behind this theorem is the following. We will construct dynamical systems that look like suspended flows over countable interval exchange transformations, so that the local behavior is that of a translation, whereas the entropy comes from infinity.

\begin{remark} The Riemannian metric $g$ of our construction is not complete. We know that we can always find a complete Riemannian metric conformal to $g$ (see \cite[Theorem 1]{MR0133785}), but we cannot ensure that the Lyapunov exponents associated to this one remain equal to zero.
\end{remark}

In section 2 we give some background on entropy, Lyapunov exponents and countable interval exchange transformations. Section 3 is devoted to the construction of the manifold $M$, the diffeomorphism $f$ and the measure $\mu$, that are used to prove Theorem \ref{T1} in section 4.

\section{Preliminaries}

\subsection{Measure-theoretic entropy} Let $(X,\mu)$ be a probability space and $T:X\vers X$ a measurable transformation. Let $\mathcal{P}$ be a finite measurable partition of $X$. The entropy of $\mathcal{P}$ with respect to $\mu$, denoted by $H_{\mu}(\mathcal{P})$, is defined as
$$H_{\mu}(\mathcal{P})=-\sum_{P\in\mathcal{P}}\mu(P)\log\mu(P).$$
For all $n\geq 0$ define the partition $\mathcal{P}^{n}$ as the measurable partition consisting of all possible intersections of elements of $T^{-i}\mathcal{P}$, for all $i=0,...,n-1$. The entropy of $T$ with respect to the partition $\mathcal{P}$ is then defined as the limit
$$h_{\mu}(T,\mathcal{P})=\lim_{n\vers\infty} \frac{1}{n}H_{\mu}(\mathcal{P}^{n}).$$
The \textit{measure-theoretic entropy of} $T$, with respect to $\mu$, is the supremum of the entropies $h_{\mu}(T,\mathcal{P})$ over all measurable finite partitions $\mathcal{P}$ of $X$, i.e.
$$h_{\mu}(T)=\sup_{\mathcal{P} \text{ finite}} h_{\mu}(T,\mathcal{P}).$$
We recall that the measure-theoretic entropy is invariant under (measure) conjugation (see for instance \cite[Theorem 4.11]{MR648108}).

\subsection{Lyapunov exponents} In the study of a smooth dynamical system it is natural to linearize the dynamics, and the notion of Lyapunov exponents is particularly relevant. Let $(M,g)$ be a Riemannian manifold and $f:M\vers M$ a $C^{1}$-map. For $x\in M$, let $\|\cdot\|_{x}$ denote the Riemannian norm induced by $g$ on $T_{x}M$. The point $x$ is said to be (Lyapunov-Perron) \textit{regular} if there exist numbers $\{\lambda_{i}(x)\}_{i=1}^{s(x)}$, called \textit{Lyapunov exponents}, and a decomposition of the tangent space at $x$ into $T_{x}M=\bigoplus_{i=1}^{s(x)}E_{i}(x)$ such that for every tangent vector $v\in E_{i}(x)\setminus\{0\}$, we have
$$\lim_{n\vers\pm\infty}\frac{1}{n}\log \|d_{x}f^{n}v\|_{f^{n}x}=\lambda_{i}(x).$$
Let $\Lambda$ be the set of regular points. If $x\in\Lambda$ and $\lambda_{i}(x)$ is a positive Lyapunov exponent, locally the action of $d_{x}f$ in the direction of $E_{i}(x)$ is expanding. On the other hand, if $\lambda_{i}(x)$ is a negative Lyapunov exponent, locally the action of $d_{x}f$ in the direction of $E_{i}(x)$ is contracting.

By a theorem of Oseledec (\cite{MR0240280},\cite{MR876081}), if $\mu$ is an $f$-invariant probability measure on $M$ such that $\log^{+}\|df\|$ and $\log^{+}\|df^{-1}\|$ are $\mu$-integrable, the set $\Lambda$ is a set of $\mu$-full measure. When $M$ is compact and $f$ is $C^{1}$ these assumptions are always satisfied. In particular, by Oseledec's Theorem, the set $\Lambda$ is a set of full measure for all $f$-invariant probability measures over $M$.

\begin{theorem}[Ruelle] Let $M$ be a compact Riemannian manifold and $f:M\vers M$ a $C^{1}$-diffeomorphism. Then, for every probability $f$-invariant measure $\mu$ on $M$, we have
$$h_{\mu}(f)\leq \int \chi^{+}(x)d\mu(x).$$
\end{theorem}

\subsection{Interval exchange transformations} As said above, the main idea to construct the family of counterexamples to Ruelle's inequality is to imitate the dynamics of a countable interval exchange transformation.\\
A countable interval exchange transformation is an invertible measurable map $T:[0,1)\vers[0,1)$ satisfying the following conditions
\begin{enumerate}
  \item There is a strictly increasing sequence $\{x_{i}\}\subset [0,1)$ and a sequence $\{a_{i}\}\subset \R$ such that $x_{0}=0$, $\lim_{i\vers\infty}x_{i}=1$ and $T(x)=x+a_{i}$ for all $x\in[x_{i},x_{i+1})$;
  \item The unique accumulation point of the set $\{x_{i}+a_{i}\}\cup \{x_{i+1}+a_{i}\}$ is 1;
\end{enumerate}

Denote by $m$ the Lebesgue measure on $[0,1)$ and let $T$ be a countable interval exchange transformation. Since $T$ is piecewisely defined by translations, it preserves $m$.
We denote by $\mathcal{I}_{T}$, or simply $\mathcal{I}$, the partition of $[0,1)$ defined by the intervals $\{[x_{i},x_{i+1})\}_{i\geq 0}$. This partition satisfies the following entropy property.

\begin{prop}[Blume, \cite{MR2968212}]\label{entropiesubintervals} Let $T$ be an interval exchange transformation. If $h_{m}(T)>0$ then $H_{m}(\mathcal{I})=\infty$.
\end{prop}

Let $(X,m,T)$ be an ergodic probability dynamical system. We say that $T$ is \textit{aperiodic} if it is invertible and the set of periodic points is a set of $m$-measure equal to zero. The following theorem says that the study of aperiodic dynamical can be reduced to the study of countable interval exchange transformations.

\begin{theorem}[Arnoux-Orstein-Weiss, \cite{MR788073}]\label{Arnoux} Every aperiodic dynamical system is measurably conjugated to a countable interval exchange transformation over $[0,1)$ endowed with the Lebesgue measure.
\end{theorem}

This theorem give us plenty of useful dynamical systems to manipulate in order to construct our counterexample to Ruelle's inequality. More precisely, since the entropy is invariant by conjugation we see that for all $h\in (0,\infty]$ there exists a countable interval exchange transformation with measure-theoretic entropy (with respect to the Lebesgue measure) equal to $h$.

\section{Suspension flows as smooth dynamical systems}

The aim of this section is to construct a smooth dynamical system which looks roughly like a suspension flow over a fixed countable interval exchange transformation. Let $T:[0,1)\vers[0,1)$ be a countable interval exchange transformation and let $I=(0,1)$ be the unit interval. We will consider the family $\{I_{i}\}_{i\in \N}$ of subintervals of $I$ defined as $I_{0}=(x_{0},x_{1})$ and $I_{i}=[x_{i},x_{i+1})$ for all $i\geq 1$. By simplicity we put $l_{i}=m(I_{i})$ and $S=\{x_{i}\}_{i\in \N}$. Note that since $T$ is piecewise defined by translations, the map $T|_{I\setminus S}$ is a smooth transformation.

\subsection{Construction of the Riemannian manifold} We are going to construct a function $r:I\vers\R^{+}\cup\{+\infty\}$ such that $r|_{I\setminus S}$ is smooth and $\lim_{x\vers x_{i}}r(x)=+\infty$ for all $i\geq 0$. Moreover, this function will be constant equal to 1 on a set of large Lebesgue measure. For all $i\geq 0$, let $l_{i}$ be the length of the interval $I_{i}$ and consider a real number $0<b_{i}<l_{i}/2$. The $b_{i}$'s will be chosen more precisely in section 4. We define five subintervals of $I_{i}$ as follows:
$$I_{i,1}=]x_{i},x_{i}+b_{i}/2[, \ I_{i,2}=[x_{i}+b_{i}/2,x_{i}+b_{i}[, \ I_{i,3}=[x_{i}+b_{i},x_{i+1}-b_{i}[,$$
$$\ I_{i,4}=[x_{i+1}-b_{i},x_{i+1}-b_{i}/2[, \ I_{i,5}=[x_{i+1}-b_{i}/2,x_{i+1}[.$$
Let $\alpha:\R\vers[0,1]$ be a smooth function such that $\alpha|_{(-\infty,0]}\equiv 1$, the restriction $\alpha|_{I}$ is strictly decreasing and $\alpha|_{[1,\infty)}\equiv 0$. Let $\gamma_{i,2}:[x_{i}+b_{i}/2,x_{i}+b_{i}]\vers[0,1]$ and $\gamma_{i,4}:[x_{i+1}-b_{i},x_{i+1}-b_{i}/2]\vers[0,1]$ be defined by
$$\gamma_{i,2}(x)=(x-(x_{i}+b_{i}/2))/(b_{i}/2) \quad\mbox{and}\quad \gamma_{i,4}(x)=(x-(x_{i+1}-b_{i}))/(b_{i}/2).$$
Finally consider the function $\alpha_{i}:I_{i}\vers\R$ defined by
$$\alpha_{i}(x)=
\begin{cases}
1, & \text{if }x\in I_{i,1}\\
(\alpha\circ \gamma_{i,2})(x), & \text{if }x\in I_{i,2}\\
0, & \text{if }x\in I_{i,3}\\
1-(\alpha\circ \gamma_{i,4})(x), & \text{if }x\in I_{i,4}\\
1, & \text{if }x\in I_{i,5}.
\end{cases}
$$
Note that, for all $i\geq 0$, the function $\alpha_{i}$ is smooth by construction. In order to define the function $r$, we proceed as follows: first consider for all $i\geq 0$ the function $f_{i}:I_{i}\vers\R^{+}$ defined by
$$f_{i}(x)=
\begin{cases}
1-\log((x-x_{i})/b_{i}), & \text{if }x\in(x_{i},x_{i}+l_{i}/2] \\
1-\log((x_{i+1}-x)/b_{i}), & \text{if }x\in[x_{i}+l_{i}/2,x_{i+1}).
\end{cases}$$
The function $f_{i}$ is smooth on $(x_{i},x_{i}+l_{i}/2)$ and $(x_{i}+l_{i}/2,x_{i+1})$. The map $r_{i}:I_{i}\vers \R$ defined by $r_{i}(x)=\alpha_{i}(x)f_{i}(x)+(1-\alpha_{i}(x))$ is smooth on $I_{i}$. It is constant equal to $1$ over $I_{i,3}$ since $\alpha_{i}$ is equal to zero over $I_{i,3}$. It is equal to $f_{i}$ over $I_{i,1}\cup I_{i,5}$. Finally, define the map $r$ as $r_{i}$ over $I_{i}$ and equal to $+\infty$ otherwise.\\

\begin{center}
\begin{tikzpicture}[scale=2]
\coordinate (A) at (0,0);
\coordinate (B) at (5,0);
\coordinate (X1) at (0.1,0);
\coordinate (X2) at (0.6,0);
\coordinate (X3) at (1.1,0);
\coordinate (X4) at (3.9,0);
\coordinate (X5) at (4.4,0);
\coordinate (X6) at (4.9,0);

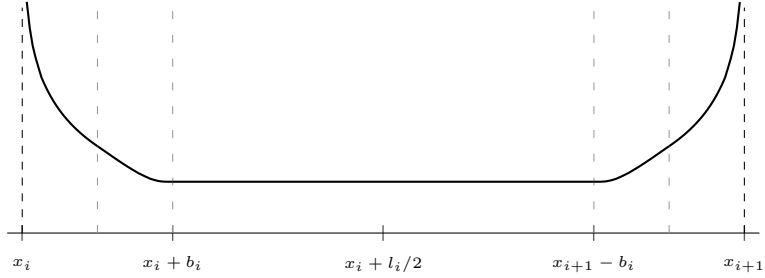
\draw (A)--(B) node[midway,above=90pt]{$\textbf{Graph of the function } r_{i}$};
\tiny
\draw  (X1) node{$|$} node[below=7.4pt]{$x_{i}$};
\draw  (X3) node{$|$} node[below=5.9pt]{$x_{i}+b_{i}$};
\draw  (X4) node{$|$} node[below=5.9pt]{$x_{i+1}-b_{i}$};
\draw  (X6) node{$|$} node[below=7.4pt]{$x_{i+1}$};
\draw  (2.5,0) node{$|$} node[below=5.9pt]{$x_{i}+l_{i}/2$};

\draw [dashed](X1)--(0.1,1.5);
\draw[color=gray!80] [loosely dashed](X2)--(0.6,1.5);
\draw[color=gray!80] [loosely dashed](X3)--(1.1,1.5);
\draw[color=gray!80] [loosely dashed](X4)--(3.9,1.5);
\draw[color=gray!80] [loosely dashed](X5)--(4.4,1.5);
\draw [dashed](X6)--(4.9,1.5);

\draw [thick][domain=0.13:0.6]plot(\x,{(1-ln((\x-0.1)))/2.93});
\draw [thick][domain=0.6:1.1]plot(\x,{(exp(1)*exp(-1/(1-((\x-0.6)/0.5)^(2)))*((1-ln((\x-0.1))))+(1-exp(1)*exp(-1/(1-((\x-0.6)/0.5)^(2)))))/2.93});
\draw [thick][domain=1.1:3.9]plot(\x,1/2.93);
\draw [thick][domain=4.4:3.9]plot(\x,{(exp(1)*exp(-1/(1-((\x-4.4)/0.5)^(2)))*((1-ln((4.9-\x))))+(1-exp(1)*exp(-1/(1-((\x-4.4)/0.5)^(2)))))/2.93});
\draw [thick][domain=4.4:4.87]plot(\x,{(1-ln((4.9-\x)))/2.93});

\end{tikzpicture}\hspace{25pt}
\end{center}

\noindent

Consider $T$ and $r$ as above. We define a topological space $M=M(T,r)$ as the quotient $I\times\R/\sim$ with the induced topology from $I\times \R$, where the equivalence relation $\sim$ is defined by $(x,r(x))\sim (T(x),-r(x))$. We denote by $\pi:I\times\R\vers M$ the canonical projection defined by this relation. By simplicity denote by $[x,y]$ the projection of the point $(x,y)\in I\times \R$ onto $M$. Our goal is to prove first the existence of a structure of smooth (noncompact) manifold on $M$, and second, that it admits a Riemannian metric. Let $M^{\ast}$ be the subset of $M$ defined by
$$M^{\ast}=\{[x,y]\in M: \ x\in ]0,1[, \ -r(T^{-1}x)< y< r(x)\})$$
and let $F$ be the subset of $M$ defined by $F=\{[x,r(x)]:x\in I\setminus S\}$.

\begin{center}
\begin{tikzpicture}[scale=4.45]
\coordinate (A) at (0,0);
\coordinate (B) at (1,0);
\coordinate (X1) at (0.2,0);
\coordinate (X2) at (0.3,0);
\coordinate (X3) at (0.6,0);
\coordinate (X4) at (0.85,0);
\coordinate (X5) at (0.9,0);
\coordinate (X6) at (0.94,0);
\coordinate (X7) at (0.97,0);

\coordinate (TX1) at (0.3,0);
\coordinate (TX2) at (0.5,0);
\coordinate (TX3) at (0.75,0);
\coordinate (TX4) at (0.85,0);
\coordinate (TX5) at (0.89,0);
\coordinate (TXBLA) at (0.9,0);

\draw (A)--(B) node[midway,above=105pt]{$\textbf{The topological space } M$};

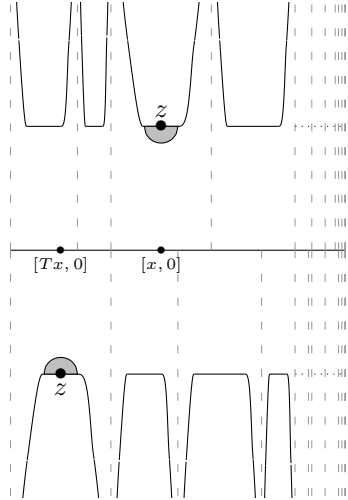
\clip (0,-2/2.7) rectangle (1,2/2.7);

\draw[color=gray!80] [loosely dashed](A)--(0.0,1);

\draw [domain=0.01:0.025]plot(\x,{(1-ln((\x-0)/0.05))/2.7});
\draw [domain=0.025:0.05]plot(\x,{((exp(1)*exp(-1/(1-((\x-(0+0.025))/(0.025))^(2))))*(1-ln((\x-0)/0.05))+(1-exp(1)*exp(-1/(1-((\x-(0+0.025))/(0.025))^(2)))))/2.7});
\draw [domain=0.05:0.15]plot(\x,1/2.7);
\draw [domain=0.15:0.175]plot(\x,{((1-exp(1)*exp(-1/(1-((\x-(0.2-0.05))/(0.025))^(2))))*(1-ln((0.2-\x)/(0.05)))+exp(1)*exp(-1/(1-((\x-(0.2-0.05))/(0.025))^(2))))/2.7});
\draw [domain=0.175:0.1901]plot(\x,{(1-ln((0.2-\x)/(0.05)))/2.7});

\draw[color=gray!80] [loosely dashed](X1)--(0.2,1);

\draw [domain=0.205:0.212]plot(\x,{(1-ln((\x-0.2)/0.025))/2.7});
\draw [domain=0.212:0.225]plot(\x,{((exp(1)*exp(-1/(1-((\x-(0.2+0.012))/(0.013))^(2))))*(1-ln((\x-0.2)/0.025))+(1-exp(1)*exp(-1/(1-((\x-(0.2+0.012))/(0.013))^(2)))))/2.7});
\draw [domain=0.225:0.275]plot(\x,1/2.7);
\draw [domain=0.275:0.288]plot(\x,{((1-exp(1)*exp(-1/(1-((\x-(0.3-0.025))/(0.013))^(2))))*(1-ln((0.3-\x)/(0.025)))+exp(1)*exp(-1/(1-((\x-(0.3-0.025))/(0.013))^(2))))/2.7});
\draw [domain=0.288:0.295]plot(\x,{(1-ln((0.3-\x)/(0.025)))/2.7});

\draw[color=gray!80] [loosely dashed](X2)--(0.3,1);

\draw [domain=0.32:0.35]plot(\x,{(1-ln((\x-0.3)/0.1))/2.7});
\draw [domain=0.35:0.4]plot(\x,{((exp(1)*exp(-1/(1-((\x-(0.3+0.05))/(0.05))^(2))))*(1-ln((\x-0.3)/0.1))+(1-(exp(1)*exp(-1/(1-((\x-(0.3+0.05))/(0.05))^(2))))))/2.7});
\draw [domain=0.4:0.5]plot(\x,1/2.7);
\draw [domain=0.5:0.55]plot(\x,{((1-exp(1)*exp(-1/(1-((\x-(0.6-0.1))/(0.05))^(2))))*(1-ln((0.6-\x)/(0.1)))+exp(1)*exp(-1/(1-((\x-(0.6-0.1))/(0.05))^(2))))/2.7});
\draw [domain=0.55:0.58]plot(\x,{(1-ln((0.6-\x)/(0.1)))/2.7});

\draw[color=gray!80] [loosely dashed](X3)--(0.6,1);

\draw [domain=0.61:0.625]plot(\x,{(1-ln((\x-0.6)/0.05))/2.7});
\draw [domain=0.625:0.65]plot(\x,{((exp(1)*exp(-1/(1-((\x-(0.6+0.025))/(0.025))^(2))))*(1-ln((\x-0.6)/0.05))+ (1-exp(1)*exp(-1/(1-((\x-(0.6+0.025))/(0.025))^(2)))) )/2.7});
\draw [domain=0.65:0.8]plot(\x,1/2.7);
\draw [domain=0.8:0.825]plot(\x,{((1-exp(1)*exp(-1/(1-((\x-(0.85-0.05))/(0.025))^(2))))*(1-ln((0.85-\x)/(0.05)))+exp(1)*exp(-1/(1-((\x-(0.85-0.05))/(0.025))^(2))))/2.7});
\draw [domain=0.825:0.8401]plot(\x,{(1-ln((0.85-\x)/(0.05)))/2.7});

\draw[color=gray!80] [loosely dashed](X4)--(0.85,1);

\draw[color=black] [dotted](0.851,1/2.7)--(1,1/2.7);
\draw[color=black] [dotted](0.851,-1/2.7)--(1,-1/2.7);

\draw[color=gray!80] [loosely dashed](X5)--(0.9,1);
\draw[color=gray!80] [loosely dashed](X6)--(0.94,1);
\draw[color=gray!80] [loosely dashed](X7)--(0.97,1);
\draw[color=gray!80] [loosely dashed](0.98,0)--(0.98,1);
\draw[color=gray!80] [loosely dashed](0.99,0)--(0.99,1);
\draw[color=gray!80] [loosely dashed](0.999,0)--(0.999,1);
\draw[color=gray!80] [loosely dashed](0.9999,0)--(0.9999,1);
\draw[color=gray!80] [loosely dashed](B)--(1,1);

\draw[color=gray!80] [loosely dashed](A)--(0.0,-1);

\draw [domain=0.02:0.05]plot(\x,{-(1-ln((\x-0)/0.1))/2.7});
\draw [domain=0.05:0.1]plot(\x,{-((exp(1)*exp(-1/(1-((\x-(0+0.05))/(0.05))^(2))))*(1-ln((\x-0)/0.1))+(1-(exp(1)*exp(-1/(1-((\x-(0+0.05))/(0.05))^(2))))))/2.7});
\draw [domain=0.1:0.2]plot(\x,-1/2.7);
\draw [domain=0.2:0.25]plot(\x,{-((1-exp(1)*exp(-1/(1-((\x-(0.3-0.1))/(0.05))^(2))))*(1-ln((0.3-\x)/(0.1)))+exp(1)*exp(-1/(1-((\x-(0.3-0.1))/(0.05))^(2))))/2.7});
\draw [domain=0.25:0.28]plot(\x,{-(1-ln((0.3-\x)/(0.1)))/2.7});

\draw[color=gray!80] [loosely dashed](TX1)--(0.3,-1);

\draw [domain=0.31:0.325]plot(\x,{-(1-ln((\x-0.3)/0.05))/2.7});
\draw [domain=0.325:0.35]plot(\x,{-((exp(1)*exp(-1/(1-((\x-(0.3+0.025))/(0.025))^(2))))*(1-ln((\x-0.3)/0.05))+(1-exp(1)*exp(-1/(1-((\x-(0.3+0.025))/(0.025))^(2)))))/2.7});
\draw [domain=0.35:0.45]plot(\x,-1/2.7);
\draw [domain=0.45:0.475]plot(\x,{-((1-exp(1)*exp(-1/(1-((\x-(0.5-0.05))/(0.025))^(2))))*(1-ln((0.5-\x)/(0.05)))+exp(1)*exp(-1/(1-((\x-(0.5-0.05))/(0.025))^(2))))/2.7});
\draw [domain=0.475:0.4901]plot(\x,{-(1-ln((0.5-\x)/(0.05)))/2.7});

\draw[color=gray!80] [loosely dashed](TX2)--(0.5,-1);

\draw [domain=0.51:0.525]plot(\x,{-(1-ln((\x-0.5)/0.05))/2.7});
\draw [domain=0.525:0.55]plot(\x,{-((exp(1)*exp(-1/(1-((\x-(0.5+0.025))/(0.025))^(2))))*(1-ln((\x-0.5)/0.05))+ (1-exp(1)*exp(-1/(1-((\x-(0.5+0.025))/(0.025))^(2)))) )/2.7});
\draw [domain=0.55:0.7]plot(\x,-1/2.7);
\draw [domain=0.7:0.725]plot(\x,{-((1-exp(1)*exp(-1/(1-((\x-(0.75-0.05))/(0.025))^(2))))*(1-ln((0.75-\x)/(0.05)))+exp(1)*exp(-1/(1-((\x-(0.75-0.05))/(0.025))^(2))))/2.7});
\draw [domain=0.725:0.7401]plot(\x,{-(1-ln((0.75-\x)/(0.05)))/2.7});

\draw[color=gray!80] [loosely dashed](TX3)--(0.75,-1);

\draw [domain=0.755:0.762]plot(\x,{-(1-ln((\x-0.75)/0.025))/2.7});
\draw [domain=0.762:0.775]plot(\x,{-((exp(1)*exp(-1/(1-((\x-(0.75+0.012))/(0.013))^(2))))*(1-ln((\x-0.75)/0.025))+(1-exp(1)*exp(-1/(1-((\x-(0.75+0.012))/(0.013))^(2)))))/2.7});
\draw [domain=0.775:0.825]plot(\x,-1/2.7);
\draw [domain=0.825:0.838]plot(\x,{-((1-exp(1)*exp(-1/(1-((\x-(0.85-0.025))/(0.013))^(2))))*(1-ln((0.85-\x)/(0.025)))+exp(1)*exp(-1/(1-((\x-(0.85-0.025))/(0.013))^(2))))/2.7});
\draw [domain=0.838:0.845]plot(\x,{-(1-ln((0.85-\x)/(0.025)))/2.7});

\draw[color=gray!80] [loosely dashed](TX4)--(0.85,-1);
\draw[color=gray!80] [loosely dashed](TX5)--(0.89,-1);
\draw[color=gray!80] [loosely dashed](TXBLA)--(0.9,-1);
\draw[color=gray!80] [loosely dashed](0.94,0)--(0.94,-1);
\draw[color=gray!80] [loosely dashed](0.98,0)--(0.98,-1);
\draw[color=gray!80] [loosely dashed](0.99,0)--(0.99,-1);
\draw[color=gray!80] [loosely dashed](0.999,0)--(0.999,-1);
\draw[color=gray!80] [loosely dashed](0.9999,0)--(0.9999,-1);
\draw[color=gray!80] [loosely dashed](B)--(1,-1);

\filldraw[draw=black, fill=gray!50]
(0.40,1/2.7)--(0.5,1/2.7)--plot[smooth,domain=0.5:0.4](\x,{1/2.7-sqrt((0.05)^2-(\x-0.45)^2)})--cycle;
\draw (0.45,1/2.7) node{$\bullet$} node[above]{$z$};

\filldraw[draw=black, fill=gray!50]
(0.1,-1/2.7)--(0.2,-1/2.7)--plot[smooth,domain=0.2:0.1](\x,{-1/2.7+sqrt((0.05)^2-(\x-0.15)^2)})--cycle;
\draw (0.15,-1/2.7) node{$\bullet$} node[below]{$z$};

\tiny
\draw (0.45,0) node{$\bullet$} node[below]{$[x,0]$};
\draw (0.15,0) node{$\bullet$} node[below]{$[Tx,0]$};

\end{tikzpicture}
\end{center}

\begin{prop}\label{ExistenceDiffStructureM} The topological space $M=M(T,r)$ admits a structure of a smooth manifold.
\end{prop}
\begin{proof} We will consider two families of local charts on $M$. For $z=[x,y]\in M^{\ast}$, let $\varepsilon>0$ be such that the Euclidean $\varepsilon$-ball centered at $(x,y)\in\R^{2}$, denoted by $B((x,y),\varepsilon)$, is contained in the set $\pi^{-1}M^{\ast}$. The local chart around $z$ is then defined by the inverse map $\psi^{e}_{z}=\pi^{-1}$ from $\pi(B((x,y),\varepsilon))$ to $B((x,y),\varepsilon)$. Such a local chart will be called \emph{first-kind local chart}. On the other hand, if $z=[\tilde{x},r(\tilde{x})]\in F$, the definition of a local chart around $z$ is more delicate. Consider $\tilde{x}\in I_{j}$ for some $j\geq 0$. Choose two real numbers $0<\varepsilon<\frac{1}{2}\min\{|\tilde{x}-x_{j}|,|\tilde{x}-x_{j+1}|\}$ and $0<\eta<\frac{1}{2}$. We define the sets $V^{\varepsilon,\eta}_{+,j}(z)$ and $V^{\varepsilon,\eta}_{-,j}(z)$ by
$$V^{\varepsilon,\eta}_{+,j}(z)=\{(x,y):|x-\tilde{x}|<\varepsilon, \ r(x)-\eta<y\leq r(x)\}$$
and
$$V^{\varepsilon,\eta}_{-,j}(z)=\{(x,y):|x-T\tilde{x}|<\varepsilon, \ -r(x)\leq y< -r(x)+\eta\}.
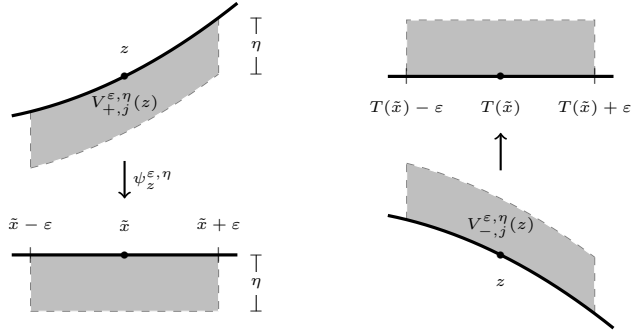
$$

\begin{center}
\begin{tikzpicture}[scale=2.5]
\tiny
\filldraw[draw=white, fill=gray!50]
(0,0)--(1,0)--(1,-0.3)--(0,-0.3)--cycle;

\draw [dashed,color=gray] (0,0)--(1,0)--(1,-0.3)--(0,-0.3)--cycle;
\draw [very thick](-0.1,0)--(1.1,0) node[midway,above=6pt]{$\tilde{x}$} node[midway]{$\bullet$};
\draw [thick](0,0) node{$|$} node[above=6pt]{$\tilde{x}-\varepsilon$};
\draw [thick](1,0) node{$|$} node[above=6pt]{$\tilde{x}+\varepsilon$};

\draw (1.2,0)--(1.2,-0.3);
\draw (1.2,0) node{$-$};
\draw (1.2,-0.3) node{$-$};
\draw (1.2,-0.15) node[fill=white]{$\eta$};

\filldraw[draw=white, fill=gray!50]
(0,0.4625)--plot[domain=0:1](\x,{0.7+((\x+0.5)^(2))/4})--(1,1.2625)--plot[domain=1:0](\x,{0.4+((\x+0.5)^(2))/4})--cycle;

\draw [dashed,color=gray][domain=0:1] plot(\x,{0.4+((\x+0.5)^(2))/4});
\draw [dashed,color=gray] (0,0.7625)--(0,0.4625);
\draw [dashed,color=gray] (1,1.2625)--(1,0.9625);
\draw [very thick][domain=-0.1:1.1] plot(\x,{0.7+((\x+0.5)^(2))/4});

\draw (1.2,0.9625)--(1.2,1.2625);
\draw (1.2,0.9625) node{$-$};
\draw (1.2,1.2625) node{$-$};
\draw (1.2,1.1125) node[fill=white]{$\eta$};

\draw [very thick](0.5,0.95) node[above=6pt]{$z$} node{$\bullet$};

\draw [very thick](0.5,0.8) node{$V^{\varepsilon,\eta}_{+,j}(z)$};

\draw [thick,sloped](0.5,0.5)--(0.5,0.3) [->];
\draw [thick](0.65,0.4) node{$\psi^{\varepsilon,\eta}_{z}$};


\filldraw[draw=white, fill=gray!50]
(2,0.1875)--plot[domain=2:3](\x,{0.25-((\x-1.5)^(2))/4})--(3,-0.0125)--plot[domain=3:2](\x,{0.55-((\x-1.5)^(2))/4})--cycle;

\draw [dashed,color=gray][domain=2:3] plot(\x,{0.55-((\x-1.5)^(2))/4});
\draw [dashed,color=gray] (2,0.1875)--(2,0.4875);
\draw [dashed,color=gray] (3,-0.0125)--(3,-0.3125);
\draw [very thick][domain=1.9:3.1] plot(\x,{0.25-((\x-1.5)^(2))/4});

\draw [very thick](2.5,0) node[below=6pt]{$z$} node{$\bullet$};

\draw [very thick](2.5,0.15) node{$V^{\varepsilon,\eta}_{-,j}(z)$};

\draw [thick,sloped](2.5,0.45)--(2.5,0.65) [->];

\filldraw[draw=white, fill=gray!50]
(2,1.25)--(3,1.25)--(3,0.95)--(2,0.95)--cycle;

\draw [dashed,color=gray] (2,1.25)--(3,1.25)--(3,0.95)--(2,0.95)--cycle;
\draw [very thick](1.9,0.95)--(3.1,0.95) node[midway,below=6pt]{$T(\tilde{x})$} node[midway]{$\bullet$};
\draw [thick](2,0.95) node{$|$} node[below=6pt]{$T(\tilde{x})-\varepsilon$};
\draw [thick](3,0.95) node{$|$} node[below=6pt]{$T(\tilde{x})+\varepsilon$};

\end{tikzpicture}
\end{center}

Note that the set $V^{\varepsilon,\eta}(z)=\pi(V^{\varepsilon,\eta}_{+,j}(z)\cup V^{\varepsilon,\eta}_{-,j}(z))$ is an open neighbourhood of $z$. Moreover, the application $\psi^{\varepsilon,\eta}_{z}:V^{\varepsilon,\eta}(z)\vers \R^{2}$, defined by
$$\psi^{\varepsilon,\eta}_{z}([x,y])=
\begin{cases}
(x,y-r(x)), & \text{if }(x,y)\in V^{\varepsilon,\eta}_{+,j}(z) \\
(x,y+r(x)), & \text{if }(Tx,y)\in V^{\varepsilon,\eta}_{-,j}(z),
\end{cases}
$$
defines a local chart around $z$ that we will call \emph{second-kind local chart}.\\
It remains to show that all transition maps between local charts are smooth. It is straightforward for two local charts of the first-kind or two local charts of second-kind. In both cases we obtain the identity as transition map. Consider a local chart $\psi_{1}$ of first-kind and a local chart $\psi_{2}$ of second-kind. The transition map $\psi_{1}\circ \psi_{2}^{-1}$ will be of the form $(x,y)\mapsto (x,y+r(x))$ or $(x,y)\mapsto (Tx,y-r(x))$ depending on the domain $V^{\varepsilon,\eta}_{+,j}(z)$ or $V^{\varepsilon,\eta}_{-,j}(z)$ that we take. If we consider now the transition map $\psi_{2}\circ \psi_{1}^{-1}$, it will be of the form $(x,y)\mapsto (x,y-r(x))$ or $(x,y)\mapsto (T^{-1}x,y-r(T^{-1}x))$ depending also on the same domains. In both cases, the transition maps are smooth since $r$ and $T$ are smooth on their respective domains.

\end{proof}

Recall that $M^{\ast}$ is the image of the set $N=\{(x,y)\in \R^{2}: \ x\in ]0,1[, \ -r(T^{-1}x)< y< r(x)\}$ under the projection map $\pi$. Since $\pi|_{N}$ is a homeomorphism, the Euclidean metric $\tilde{g}^{e}$ on $\R^{2}$ induces a natural Riemannian metric $g^{e}=(\pi|_{N}^{-1})^{\ast}\tilde{g}^{e}$ on $M^{\ast}$. This Euclidean metric $g^{e}$ cannot be extended to the whole manifold $M$ because the local charts around the points of $F$, where $r$ is not locally constant, cause distortion of the Euclidian metric.

\begin{definition} Let $M$ be a smooth manifold. Let $g^{1}$ and $g^{2}$ be two Riemannian metrics on $M$. We say that $g^{1}$ and $g^{2}$ are \textit{pointwise} equivalent if for all $p\in M$ there exists a constant $C(p)\geq 1$ such that for all $v\in T_{p}M$,
$$C(p)^{-1}\leq \frac{g^{1}_{p}(v,v)}{g^{2}_{p}(v,v)}\leq C(p).$$
\end{definition}

\begin{prop}\label{ExistenceRiemMetricM} The smooth manifold $M=M(T,r)$ admits a Riemannian metric $g$, which is pointwise equivalent to the Euclidian metric $g^{e}$ in restriction to $M^{\ast}$.
\end{prop}
\begin{proof} Let $z\in F$ and choose some $\varepsilon>0$ and $0<\delta<1/2$ such that the second-kind local chart around $z$ is well defined. Define a Riemannian metric $h^{\delta}$ on $V^{\varepsilon,\delta}(z)$ by $h^{\delta}=(\psi^{\varepsilon,\delta}_{z})^{\ast}\tilde{g}^{e}$. This Riemannian metric is well defined since the transition map between two local charts of second-kind is the identity. Let $R^{\delta}$ be the set defined by
$$R^{\delta}=\{[x,y]\in M:x\in I\setminus S, -r(T^{-1}x)<y<-r(T^{-1}x)+\delta \textmd{ or }r(x)-\delta < y \leq r(x)\}.$$
We remark that $R^{\delta}$ is the set of all the points contained in some $V^{\delta,\varepsilon}(w)$, with $w\in F$. Choose now a smooth function $\rho_{\delta}:M\vers [0,1]$ such that $\rho_{\delta}|_{M\setminus R^{\delta}}\equiv 1$, $\rho_{\delta}|_{F}\equiv 0$ and $0<\rho_{\delta}<1$ otherwise. The metric $g^{\delta}$ defined by $g^{\delta}=\rho_{\delta}g^{e}+(1-\rho_{\delta})h^{\delta}$ is by construction a Riemannian metric. It coincides with the Euclidean metric $g^{e}$ on $M\setminus R^{\delta}$.

\begin{lemma}\label{compmetrics} The Riemannian metrics $g^{e}$ and $g^{\delta}$ are pointwise equivalent on $M^{\ast}$.
\end{lemma}
\begin{proof} Let $z=[x,y]\in M^{\ast}$. We denote by $\|\cdot\|^{\delta}_{z}$ (resp. $\|\cdot\|^{e}_{z}$)\footnote{Do not confuse the norm $\|\cdot\|^{\delta}$ with a power of the norm $\|\cdot\|$.} the norm induced by $g^{\delta}$ (resp. $g^{e}$) on $T_{z}M$. Let $\tilde{z}\in F$ and $\varepsilon,\delta>0$ such that $\psi^{\varepsilon,\delta}_{\tilde{z}}$ is well defined. We denote by $\|d_{z}\psi^{\varepsilon,\delta}_{\tilde{z}}\|$ the operator norm of the map $d_{z}\psi^{\varepsilon,\delta}_{\tilde{z}}:(T_{z}M,g^{e}_{z})\vers (\R^{2},\tilde{g}^{e})$. Then, if $z\in V^{\varepsilon,\delta}_{+,j}(\tilde{z})$ in local coordinates we have
$$ d_{z}\psi^{\varepsilon,\delta}_{\tilde{z}}=\left(
  \begin{array}{cc}
    1 & 0 \\
    -r'(x) & 1 \\
  \end{array}
\right), \ d_{\psi^{\varepsilon,\delta}_{\tilde{z}}z}(\psi^{\varepsilon,\delta}_{\tilde{z}})^{-1}=\left(
  \begin{array}{cc}
    1 & 0 \\
    r'(x) & 1 \\
  \end{array}
\right).$$
On the other hand, if $z\in V^{\varepsilon,\delta}_{-,j}(\tilde{z})$, in local coordinates we have
$$ d_{z}\psi^{\varepsilon,\delta}_{\tilde{z}}=\left(
  \begin{array}{cc}
    1 & 0 \\
    r'(T^{-1}x) & 1 \\
  \end{array}
\right), \ d_{\psi^{\varepsilon,\delta}_{\tilde{z}}z}(\psi^{\varepsilon,\delta}_{\tilde{z}})^{-1}=\left(
  \begin{array}{cc}
    1 & 0 \\
    -r'(T^{-1}x) & 1 \\
  \end{array}
\right).$$

Recall that in $\R^{n}$ the Euclidian norm $\|\cdot\|_{2}$ is comparable with the maximum norm $\|\cdot\|_{\infty}$ by
$$\|\cdot\|_{\infty}\leq \|\cdot\|_{2}\leq n\|\cdot\|_{\infty}.$$
Since $M$ is a $2$-dimensional manifold, we can check that  $1\leq \|d_{z}\psi^{\varepsilon,\delta}_{\tilde{z}}\|\leq 2(1+|r'(x)|)$ and $1\leq \|d_{\psi^{\varepsilon,\delta}_{\tilde{z}}z}(\psi^{\varepsilon,\delta}_{\tilde{z}})^{-1}\|\leq 2(1+|r'(x)|)$. Thus, for all $v\in T_{z}M$, we have
\begin{eqnarray*}
(\|v\|^{\delta}_{z})^{2}&=& g^{\delta}_{z}(v,v)= \rho_{\delta}(z)g^{e}_{z}(v,v)+(1-\rho_{\delta})h^{\delta}(v,v)\\
&=& \rho_{\delta}(z)g^{e}_{z}(v,v)+(1-\rho_{\delta})g^{e}(d_{z}\psi^{\varepsilon,\delta}_{\tilde{z}}(v),d_{z}\psi^{\varepsilon,\delta}_{\tilde{z}}(v))\\
&\leq& \rho_{\delta}(z)g^{e}_{z}(v,v)+(1-\rho_{\delta}) \|d_{z}\psi^{\varepsilon,\delta}_{\tilde{z}}\|^{2}g^{e}_{z}(v,v)\\
&\leq& \|d\psi^{\varepsilon,\delta}_{\tilde{z}}\|^{2}(\|v\|^{e}_{z})^{2},
\end{eqnarray*}
and
\begin{eqnarray*}
(\|v\|^{e}_{z})^{2}&=& (\|(d_{z}\psi^{\varepsilon,\delta}_{\tilde{z}})^{-1}d_{z}\psi^{\varepsilon,\delta}_{\tilde{z}}(v)\|^{e}_{z})^{2}\\
&\leq& \|(d_{z}\psi^{\varepsilon,\delta}_{\tilde{z}})^{-1}\|^{2}(\|d_{z}\psi^{\varepsilon,\delta}_{\tilde{z}}(v)\|^{e}_{z})^{2}\\
&=& \|(d_{z}\psi^{\varepsilon,\delta}_{\tilde{z}})^{-1}\|^{2}(\|v\|^{\delta}_{z})^{2}.
\end{eqnarray*}
For $z=[x,y]\in M^{\ast}$ define $C(z)$ by $C(z)=(2+2|r'(x)|)$. It follows that for all $v\in T_{z}M$, we have
\begin{gather}\label{comparaisonmetriques}
C(z)^{-1}\|v\|^{e}_{z}\leq \|v\|^{\delta}_{z} \leq C(z)\|v\|^{e}_{z}.
\end{gather}
\end{proof}
\noindent
This concludes the proof of Proposition \ref{ExistenceRiemMetricM}.
\end{proof}

Observe that the constant $C$ introduced in the proof of Lemma \ref{compmetrics} is not optimal. In fact, when $r$ is locally constant, both metrics locally coincide.

\subsection{The diffeomorphism} From now on $(M,g^{\delta})$ is the Riemannian manifold constructed above. The \emph{suspension flow} $(\phi^{t})$ on $M$ is defined as follows. For all $t\in \R$ we define $\phi^{t}:M\vers M$ by $\phi^{t}([x,y])=[x,y+t]$. The unit map $\phi^{1}=\phi$ of the suspension flow satisfies
$$\phi([x,y])=
\begin{cases}
[x,y+1] & \text{if }y+1<r(x) \\
[Tx,y+1-2r(x)] & \text{if }y+1\geq r(x).
\end{cases}
$$

\begin{center}
\begin{tikzpicture}[scale=4]
\coordinate (A) at (0,0);
\coordinate (B) at (1,0);
\coordinate (X1) at (0.2,0);
\coordinate (X2) at (0.3,0);
\coordinate (X3) at (0.6,0);
\coordinate (X4) at (0.85,0);
\coordinate (X5) at (0.9,0);
\coordinate (X6) at (0.94,0);
\coordinate (X7) at (0.97,0);

\coordinate (TX1) at (0.3,0);
\coordinate (TX2) at (0.5,0);
\coordinate (TX3) at (0.75,0);
\coordinate (TX4) at (0.85,0);
\coordinate (TX5) at (0.89,0);
\coordinate (TXBLA) at (0.9,0);

\draw [color=white](A)--(B) node[color=black,midway,below=90pt]{$\textbf{The flow } (\phi^{t})$};

\clip (0,-2/2.7) rectangle (1,2/2.7);

\draw [very thick][domain=0.01:0.025]plot(\x,{(1-ln((\x-0)/0.05))/2.7});
\draw [very thick][domain=0.025:0.05]plot(\x,{((exp(1)*exp(-1/(1-((\x-(0+0.025))/(0.025))^(2))))*(1-ln((\x-0)/0.05))+(1-exp(1)*exp(-1/(1-((\x-(0+0.025))/(0.025))^(2)))))/2.7});
\draw [very thick][domain=0.05:0.15]plot(\x,1/2.7);
\draw [very thick][domain=0.15:0.175]plot(\x,{((1-exp(1)*exp(-1/(1-((\x-(0.2-0.05))/(0.025))^(2))))*(1-ln((0.2-\x)/(0.05)))+exp(1)*exp(-1/(1-((\x-(0.2-0.05))/(0.025))^(2))))/2.7});
\draw [very thick][domain=0.175:0.1901]plot(\x,{(1-ln((0.2-\x)/(0.05)))/2.7});

\draw [very thick][domain=0.205:0.212]plot(\x,{(1-ln((\x-0.2)/0.025))/2.7});
\draw [very thick][domain=0.212:0.225]plot(\x,{((exp(1)*exp(-1/(1-((\x-(0.2+0.012))/(0.013))^(2))))*(1-ln((\x-0.2)/0.025))+(1-exp(1)*exp(-1/(1-((\x-(0.2+0.012))/(0.013))^(2)))))/2.7});
\draw [very thick][domain=0.225:0.275]plot(\x,1/2.7);
\draw [very thick][domain=0.275:0.288]plot(\x,{((1-exp(1)*exp(-1/(1-((\x-(0.3-0.025))/(0.013))^(2))))*(1-ln((0.3-\x)/(0.025)))+exp(1)*exp(-1/(1-((\x-(0.3-0.025))/(0.013))^(2))))/2.7});
\draw [very thick][domain=0.288:0.295]plot(\x,{(1-ln((0.3-\x)/(0.025)))/2.7});

\draw [very thick][domain=0.32:0.35]plot(\x,{(1-ln((\x-0.3)/0.1))/2.7});
\draw [very thick][domain=0.35:0.4]plot(\x,{((exp(1)*exp(-1/(1-((\x-(0.3+0.05))/(0.05))^(2))))*(1-ln((\x-0.3)/0.1))+(1-(exp(1)*exp(-1/(1-((\x-(0.3+0.05))/(0.05))^(2))))))/2.7});
\draw [very thick][domain=0.4:0.5]plot(\x,1/2.7);
\draw [very thick][domain=0.5:0.55]plot(\x,{((1-exp(1)*exp(-1/(1-((\x-(0.6-0.1))/(0.05))^(2))))*(1-ln((0.6-\x)/(0.1)))+exp(1)*exp(-1/(1-((\x-(0.6-0.1))/(0.05))^(2))))/2.7});
\draw [very thick][domain=0.55:0.58]plot(\x,{(1-ln((0.6-\x)/(0.1)))/2.7});

\draw [very thick][domain=0.61:0.625]plot(\x,{(1-ln((\x-0.6)/0.05))/2.7});
\draw [very thick][domain=0.625:0.65]plot(\x,{((exp(1)*exp(-1/(1-((\x-(0.6+0.025))/(0.025))^(2))))*(1-ln((\x-0.6)/0.05))+ (1-exp(1)*exp(-1/(1-((\x-(0.6+0.025))/(0.025))^(2)))) )/2.7});
\draw [very thick][domain=0.65:0.8]plot(\x,1/2.7);
\draw [very thick][domain=0.8:0.825]plot(\x,{((1-exp(1)*exp(-1/(1-((\x-(0.85-0.05))/(0.025))^(2))))*(1-ln((0.85-\x)/(0.05)))+exp(1)*exp(-1/(1-((\x-(0.85-0.05))/(0.025))^(2))))/2.7});
\draw [very thick][domain=0.825:0.8401]plot(\x,{(1-ln((0.85-\x)/(0.05)))/2.7});

\draw [very thick][domain=0.02:0.05]plot(\x,{-(1-ln((\x-0)/0.1))/2.7});
\draw [very thick][domain=0.05:0.1]plot(\x,{-((exp(1)*exp(-1/(1-((\x-(0+0.05))/(0.05))^(2))))*(1-ln((\x-0)/0.1))+(1-(exp(1)*exp(-1/(1-((\x-(0+0.05))/(0.05))^(2))))))/2.7});
\draw [very thick][domain=0.1:0.2]plot(\x,-1/2.7);
\draw [very thick][domain=0.2:0.25]plot(\x,{-((1-exp(1)*exp(-1/(1-((\x-(0.3-0.1))/(0.05))^(2))))*(1-ln((0.3-\x)/(0.1)))+exp(1)*exp(-1/(1-((\x-(0.3-0.1))/(0.05))^(2))))/2.7});
\draw [very thick][domain=0.25:0.28]plot(\x,{-(1-ln((0.3-\x)/(0.1)))/2.7});

\draw [very thick][domain=0.31:0.325]plot(\x,{-(1-ln((\x-0.3)/0.05))/2.7});
\draw [very thick][domain=0.325:0.35]plot(\x,{-((exp(1)*exp(-1/(1-((\x-(0.3+0.025))/(0.025))^(2))))*(1-ln((\x-0.3)/0.05))+(1-exp(1)*exp(-1/(1-((\x-(0.3+0.025))/(0.025))^(2)))))/2.7});
\draw [very thick][domain=0.35:0.45]plot(\x,-1/2.7);
\draw [very thick][domain=0.45:0.475]plot(\x,{-((1-exp(1)*exp(-1/(1-((\x-(0.5-0.05))/(0.025))^(2))))*(1-ln((0.5-\x)/(0.05)))+exp(1)*exp(-1/(1-((\x-(0.5-0.05))/(0.025))^(2))))/2.7});
\draw [very thick][domain=0.475:0.4901]plot(\x,{-(1-ln((0.5-\x)/(0.05)))/2.7});

\draw [very thick][domain=0.51:0.525]plot(\x,{-(1-ln((\x-0.5)/0.05))/2.7});
\draw [very thick][domain=0.525:0.55]plot(\x,{-((exp(1)*exp(-1/(1-((\x-(0.5+0.025))/(0.025))^(2))))*(1-ln((\x-0.5)/0.05))+ (1-exp(1)*exp(-1/(1-((\x-(0.5+0.025))/(0.025))^(2)))) )/2.7});
\draw [very thick][domain=0.55:0.7]plot(\x,-1/2.7);
\draw [very thick][domain=0.7:0.725]plot(\x,{-((1-exp(1)*exp(-1/(1-((\x-(0.75-0.05))/(0.025))^(2))))*(1-ln((0.75-\x)/(0.05)))+exp(1)*exp(-1/(1-((\x-(0.75-0.05))/(0.025))^(2))))/2.7});
\draw [very thick][domain=0.725:0.7401]plot(\x,{-(1-ln((0.75-\x)/(0.05)))/2.7});

\draw [very thick][domain=0.755:0.762]plot(\x,{-(1-ln((\x-0.75)/0.025))/2.7});
\draw [very thick][domain=0.762:0.775]plot(\x,{-((exp(1)*exp(-1/(1-((\x-(0.75+0.012))/(0.013))^(2))))*(1-ln((\x-0.75)/0.025))+(1-exp(1)*exp(-1/(1-((\x-(0.75+0.012))/(0.013))^(2)))))/2.7});
\draw [very thick][domain=0.775:0.825]plot(\x,-1/2.7);
\draw [very thick] [domain=0.825:0.838]plot(\x,{-((1-exp(1)*exp(-1/(1-((\x-(0.85-0.025))/(0.013))^(2))))*(1-ln((0.85-\x)/(0.025)))+exp(1)*exp(-1/(1-((\x-(0.85-0.025))/(0.013))^(2))))/2.7});
\draw [very thick][domain=0.838:0.845]plot(\x,{-(1-ln((0.85-\x)/(0.025)))/2.7});

\draw[color=gray!80] [loosely dashed](0,-1)--(0,1);
\draw[color=gray!80] [loosely dashed](1,-1)--(1,1);

\draw (0.85,1/2.7) node{$\cdot$};
\draw (0.9,1/2.7) node{$\cdot$};
\draw (0.95,1/2.7) node{$\cdot$};

\draw (0.87,-1/2.7) node{$\cdot$};
\draw (0.92,-1/2.7) node{$\cdot$};
\draw (0.97,-1/2.7) node{$\cdot$};

\draw (0.62,0.1) node{$\bullet$} node[below]{$z$};
\draw (0.83,-0.15) node{$\bullet$} node[above]{$\phi^{t}(z)$};

\draw [thick,densely dotted](0.62,0.1)--(0.62,0.685) node[midway]{$\uparrow$};
\draw [thick,densely dotted](0.52,-0.685)--(0.52,0.39) node[midway]{$\uparrow$};
\draw [thick,densely dotted](0.22,-0.388)--(0.22,1/2.7+0.03) node[midway]{$\uparrow$};
\draw [thick,densely dotted](0.83,-1/2.7-0.03)--(0.83,-0.15) node[midway,]{$\uparrow$};

\end{tikzpicture}
\end{center}

\noindent\\

\begin{prop}\label{diffunit} The map $\phi:M\vers M$ is smooth.
\end{prop}
\begin{proof} Recall that the map $\phi$ is smooth if, for all local charts $\psi_{\alpha}$ and $\psi_{\beta}$, the map $\phi_{\alpha,\beta}=\psi_{\alpha}\circ \phi \circ \psi^{-1}_{\beta}$ is smooth whenever it is defined. If $\psi_{\alpha}$ and $\psi_{\beta}$ are local charts of first-kind, the map $\phi_{\alpha,\beta}$ is equal to $(x,y)\mapsto (x,y+1)$ or $(x,y)\mapsto (Tx,y+1-2r(x))$. If $\psi_{\alpha}$ is a first-kind local chart and $\psi_{\beta}$ is a second-kind local chart, the map $\phi_{\alpha,\beta}$ is equal to $(x,y)\mapsto (Tx,y+1-r(x))$. If $\psi_{\alpha}$ is a second-kind local chart and $\psi_{\beta}$ is a first-kind local chart, the map $\phi_{\alpha,\beta}$ is equal to $(x,y)\mapsto (x,y+1-r(x))$. Since $r\geq 1$, there are no more possibilities for the map $\phi_{\alpha,\beta}$. The regularity of $r$ and $T$ implies the conclusion of the proposition.
\end{proof}

Let $\|d\phi\|^{\delta}$ be the operator norm of $d\phi$ with respect to the norm $g^{\delta}$. When $z\in M^{\ast}\cap \phi^{-1}(M^{\ast})$, consider also the operator norm $\|d_{z}\phi\|^{e}$ of $d\phi$ with respect to the Euclidean norm $g^{e}$. Next proposition says that we can compare these operator norms by an explicit ``nice'' function. This fact will be fundamental in section 4.

\begin{prop}\label{OselHyp} There exists an explicit measurable function $\beta:M^{\ast}\cap \phi^{-1}(M^{\ast})\vers\R^{+}$, defined in (\ref{betadez}), such that for all $z\in M^{\ast}\cap \phi^{-1}(M^{\ast})$, we have
\begin{gather}\label{InOp}
\|d_{z}\phi\|^{\delta}\leq \beta(z)\|d_{z}\phi\|^{e}.
\end{gather}
\end{prop}
\begin{proof} Using computations in the proof Lemma \ref{compmetrics} for all $z\in M^{\ast}\cap \phi^{-1}(M^{\ast})$ and for all $v\in T_{z}M$, we obtain
\begin{eqnarray*}
\|d_{z}\phi(v)\|^{\delta}_{\phi z}&\leq& C(\phi(z))\|d_{z}\phi(v)\|^{e}_{\phi z}\\
&\leq& C(\phi(z))\|d_{z}\phi\|^{e}\|v\|^{e}_{z}\\
&\leq& C(\phi(z))C(z)\|d_{z}\phi\|^{e}\|v\|^{\delta}_{z}\\
&\leq& (2+2|r'(x)|)\max\{(2+2|r'(x)|),(2+2|r'(Tx)|)\}\|d_{z}\phi\|^{e}\|v\|^{\delta}_{z}.
\end{eqnarray*}
Recall that $g^{\delta}$ and $g^{e}$ coincide on $M\setminus R^{\delta}$. Let $K^{\delta}$ be the set $K_{\delta}=\{[x,y]:-r(T^{-1}x)+\delta<y<r(x)-(1+\delta)\}$. Then, for $\beta(z)$ defined by
\begin{gather}\label{betadez}
\beta(z)=(2+2|r'(x)|)\max\left\{(2+2|r'(x)|),(2+2|r'(Tx)|)\right\}\textbf{1}_{M\setminus K_{\delta}}(z)+\textbf{1}_{K_{\delta}}(z)
\end{gather}
we conclude the proof of Proposition \ref{OselHyp}.
\end{proof}

Remark that, for $z=[x,y]\in M^{\ast}\cap \phi^{-1}(M^{\ast})$, the differential $d_{z}\phi$ is represented in local coordinates by the identity matrix if $y+1<r(x)$. If $y+1>r(x)$, with $x\in I\setminus \{x_{i}\}$, the differential $d_{z}\phi$ is represented in local coordinates by the matrix
$$
d_{z}\phi= \left(
\begin{array}{cc}
 1 & 0 \\
-2r'(x) & 1
\end{array}
\right).
$$

\subsection{A finite invariant measure} There is a natural $(\phi^{t})$-invariant measure $\tilde{\mu}$ on $M$ defined as follows. For a Borel set $A\subset M$, define $\tilde{\mu}(A)$ as
$$\tilde{\mu}(A)=\int_{0}^{1}\int_{-r(T^{-1}x)}^{r(x)} \textbf{1}_{\pi^{-1}(A)}(x,y)dydx.$$
Since the Lebesgue measure $m$ is invariant by translation and $\phi^{t}$ acts by translations over the vertical lines, the measure $\tilde{\mu}$ is $(\phi^{t})$-invariant.

\begin{prop}\label{mufinite} The measure $\tilde{\mu}$ is finite.
\end{prop}
\begin{proof} It is enough to prove that $\int_{I}r(x)dm$ is finite. Recall that on $I_{i,1}\cup I_{i,2}$ we have $r(x)\leq 2-\log((x-x_{i})/b_{i})$, whereas on $I_{i,4}\cup I_{i,5}$ we have $r(x)\leq 2-\log((x_{i+1}-x)/b_{i})$. Since $r|_{I_{i,3}}\equiv 1$, we obtain
\begin{eqnarray*}
\int_{0}^{1}\int_{0}^{r(x)}1 dydx &=& \sum_{i\in\N}\int_{I_{i}}r(x) dx\\
&\leq& \sum_{i\in\N}\left[2\left(\int_{0}^{b_{i}}2-\log(x/b_{i}) dx\right)\right]+l_{i}-2b_{i}\\
&=& \sum_{i\in\N} 4b_{i}+l_{i}\\
&\leq& 5.
\end{eqnarray*}
\end{proof}
\begin{definition} We define the $\phi$-invariant probability measure $\mu$ on $M$ as the normalized measure of $\tilde{\mu}$.
\end{definition}

\section{Proof of Theorem \ref{T1}}

In this section we will consider an arbitrary aperiodic countable interval exchange transformation $(I,m,T)$ of entropy $h\in (0,\infty]$. We consider the Riemannian manifold $(M,g^{\delta})$ constructed in section 3, the $C^{\infty}$-map $\phi=\phi^{1}$ defined as the time-one map of the suspension flow and the $\phi$-invariant probability measure $\mu$ as in the previous section. We also consider the Riemannian metric $g^{e}$ defined on the subset $M^{\ast}$ of $M$.

\subsection{Some technical lemmas} The aim of this subsection is to prove that under some additional condition on the roof function $r$, easy to ensure, the assumptions of Oseledec's Theorem for $d\phi$ and $d\phi^{-1}$ holds.

\begin{lemma}\label{lemma0} Define the function $h$ by
$$h(x)= \begin{cases}
2+2|r'(x)| & \text{if }x\in  I\setminus S\\
0 & \text{otherwise.}
\end{cases}
$$
If $-\sum_{i\geq 0}b_{i}\log b_{i}<\infty$, then $\log^{+}(h)$ is $m$-integrable on I.
\end{lemma}
\begin{proof}
It is sufficient to prove that $x\mapsto \log(1+|r'(x)|)$ is $m$-integrable on $I$ since
$$h(x)\leq \log(2)+\log(1+|r'(x)|).$$
Recall that, on the interval $I_{i}$, the roof function is defined by $r(x)=\alpha_{i}(x)f_{i}(x)+(1-\alpha_{i}(x))$. It follows that $|r'(x)|\leq |\alpha'_{i}(x)||f_{i}(x)-1|+|f'_{i}(x)|$. Thus, we have
$$|r'(x)|\leq \begin{cases}
b_{i}/(x-x_{i}) & \text{if }x\in  I_{i,1}\\
C/b_{i}  & \text{if }x\in I_{i,2}\\
0 & \text{if }x\in I_{i,3}\\
C/b_{i} & \text{if }x\in I_{i,4} \\
b_{i}/(x_{i+1}-x) & \text{if }x\in I_{i,5},
\end{cases}
$$
where $C\geq 1$ is a constant depending on $\sup_{x\in\R}|\alpha'(x)|<\infty$. In particular, we obtain
\begin{eqnarray*}
\int_{0}^{1}\log^{+} (1+|r'(x)|)dx &=& \sum_{i=0}^{\infty}\sum_{k=1}^{5}\int_{I_{i,k}}\log(1+|r'(x)|)dx\\
&\leq&  \sum_{i=0}^{\infty} \Bigg( \int_{I_{i,1}}\log \left(1+\frac{b_{i}}{x-x_{i}}\right)dx\\
& &+\int_{I_{i,2}}\log \left(1+\frac{C}{b_{i}}\right)dx+\int_{I_{i,3}}\log(1)dx\\
& &+\int_{I_{i,4}}\log \left(1+\frac{C}{b_{i}}\right)dx+\int_{I_{i,5}}\log \left(1+\frac{b_{i}}{x_{i+1}-x}\right)dx\Bigg) \\
&\leq& \sum_{i=0}^{\infty}(3+\log(2C))l_{i}-b_{i}\log(b_{i})\\
&\leq& 3+\log(2C)-\sum_{i=0}^{\infty}b_{i}\log(b_{i})\\
&<&+\infty.
\end{eqnarray*}
\end{proof}
The above proof shows that the assumption $-\sum_{i\geq 0}b_{i}\log b_{i}<\infty$ is crucial. In particular, if the interval exchange transformation has positive entropy $h>0$, it follows from Proposition \ref{entropiesubintervals} that $-\sum_{i=0}^{\infty}l_{i}\log l_{i}=\infty$, so that we cannot choose $b_{i}$ uniformly proportional to $l_{i}$ if we want $-\sum_{i\geq 0}b_{i}\log b_{i}<\infty$.

\begin{lemma}\label{lemma1} If $-\sum_{i\geq 0}b_{i}\log b_{i}<\infty$, then
$$\int \log^{+}\|d\phi\|^{e}d\mu<\infty \quad\mbox{and}\quad \int \log^{+}\|d\phi^{-1}\|^{e}d\mu<\infty.$$
\end{lemma}
\begin{proof}
Notice that the set $M^{\ast}\cap \phi^{-1}(M^{\ast})$ is a set of full $\mu$-measure, so that we can suppose that $z\in M^{\ast}\cap \phi^{-1}(M^{\ast})$. In local coordinates the differential application $d_{z}f$ is represented by the identity if $z\in \{[x,y]\in M:-r(T^{-1}x)< y < r(x)-1\}$ and the matrix
$$
\left(
\begin{array}{cc}
 1 & 0 \\
-2r'(x) & 1
\end{array}
\right),
$$
if $z\in \{[x,y]:r(x)-1<y<r(x)\}$. In particular, we have $\|d_{z}\phi\|^{e}\leq 2+2|r'(x)|$, so that we obtain
\begin{eqnarray*}
\int\log^{+}\|d\phi\|^{e}d\tilde{\mu}&=& \int_{0}^{1}\int_{r(x)-1}^{r(x)}\log^{+}\|d_{z}\phi\|^{e}dydx\\
&\leq& \int_{0}^{1}\int_{r(x)-1}^{r(x)} \log^{+} (2+2|r'(x)|)dydx\\
&=& \int \log^{+}(h) dm \\
&<& +\infty.
\end{eqnarray*}
Lemma \ref{lemma0} implies that the last integral is finite. The proof of the finiteness for the second integral is similar.
\end{proof}

The conclusion of Lemma \ref{lemma1} involves the norm $\|d\phi\|^{e}$, but we need to work with $\|d\phi\|^{\delta}$.

\begin{lemma}\label{lemma2} If $-\sum_{i\geq 0}b_{i}\log b_{i}<\infty$, then
$$\int \log^{+}\|d\phi\|^{\delta}d\mu<\infty \quad\mbox{and}\quad \int \log^{+}\|d\phi^{-1}\|^{\delta}d\mu<\infty.$$
\end{lemma}
\begin{proof}
From Inequality $(\ref{InOp})$, we obtain
\begin{eqnarray*}
\int\log^{+}\|d_{z}\phi\|^{\delta}d\mu(z)&\leq& \int\log^{+}(\beta(z)\|d_{z}\phi\|^{e})d\mu(z)\\
&\leq& \int\log^{+}\|d_{z}\phi\|^{e}d\mu(z) + \int\log^{+}\beta(z) d\mu(z).
\end{eqnarray*}
The first integral is finite thanks to Lemma \ref{lemma1}. For the second integral, we have
\begin{eqnarray*}
\int\log^{+}\beta d\tilde{\mu}&=&\int_{0}^{1}\int_{-r(T^{-1}x)}^{-r(T^{-1}x)+\delta}\log^{+}\beta([x,y])dydx\\
& &+ \int_{0}^{1}\int_{r(x)-(1+\delta)}^{r(x)}\log^{+}\beta([x,y])dydx\\
&\leq& (1+2\delta)\int_{0}^{1} 2\log(2+2|r'(x)|)+\log(2+2|r'(Tx)|)dx\\
&=& 3(1+2\delta)\int \log^{+}(h) dm.
\end{eqnarray*}
The last equality follows from the fact that $T$ preserves the Lebesgue measure $dx$ on $(0,1)$. Lemma \ref{lemma0} allows to conclude the proof of Lemma \ref{lemma2} for $\|d\phi\|^{\delta}$. The proof of the finiteness for the second integral is similar.
\end{proof}

\subsection{The case $h=\infty$}

Our initial goal is to find a counterexample to Ruelle's inequality for a diffeomorphism of a noncompact manifold. Suppose that the measure-theoretic entropy of $T$ with respect to $m$ is infinite. Since the roof function $r$ is $m$-integrable, it follows from Abramov's formula that the entropy of the suspension flow $h_{\mu}(\phi)$ is infinite. Choose the $b_{i}$'s such that $-\sum_{i\geq 0}b_{i}\log b_{i}<\infty$. Lemma \ref{lemma2} above ensures that Oseledec's Theorem applies, so that there exists almost-everywhere Lyapunov exponents satisfying
$$\int\chi^{+}d\mu <\infty.$$
In particular, the measure-theoretic entropy of $\phi$ with respect to $\mu$ is greater than the sum of the positive Lyapunov exponents. This contradicts Ruelle's inequality.

\subsection{Computation of Lyapunov exponents} In order to prove that the Lyapunov exponents for $g^{\delta}$ are $\mu$-a.e. equal to zero, we will first calculate the Lyapunov exponents for the Riemannian metric $g^{e}$ on $M^{\ast}$. As $M^{\ast}$ is not a $\phi$-invariant set, we actually have to work on the set $\bigcap_{k\in\Z}\phi^{k}(M^{\ast})$, which is a set of full $\mu$-measure. The key to compute the Lyapunov exponents for $g^{e}$ will be Lemma \ref{lemmaAaronson} ((\cite[Proposition 2.3.1]{MR1450400})) below. This lemma needs the ergodicity of the map $\phi$, which is not necessarily satisfied. In fact, since $T$ is ergodic, the flow $(\phi^{t})$ is obviously ergodic. This does not imply that all maps $\phi^{t}$, for fixed $t$, are ergodic. But there always exist many infinitely $t\in \R$ such that $\phi^{t}$ is ergodic (see \cite[Theorem 3.2]{MR536988}).

Let $\tau$ be an ergodic time for the flow $(\phi^{t})$. The definition of Lyapunov exponents implies that every Lyapunov exponent for the map $\phi^{\tau}$ is of the form $\tau\lambda_{i}$, where $\lambda_{i}$ is a Lyapunov exponent for $\phi=\phi^{1}$. In particular, if all Lyapunov exponents of $\phi^{\tau}$ are equal to zero, the same holds for the Lyapunov exponents of $\phi$. Thus, we can suppose without loss of generality that $\phi$ is an ergodic map. We will show that its associated Lyapunov exponents are equal to zero.

\begin{prop}[Aaronson]\label{lemmaAaronson} Let $(X,\mathcal{B},m)$ a Lebesgue space where $m$ is a probability measure. Suppose that $T:X\vers X$ is an ergodic transformation preserving $m$. If $h:X\vers\R$ is a measurable function such that $\int \log^{+}(|h|)dm<\infty$, we have
$$\lim_{n\vers\infty} \frac{1}{n}\log^{+} \left|\sum_{i=0}^{n-1}h(T^{i}x) \right|=0$$
for $m$-almost every $x\in X$.
\end{prop}

\noindent
\textit{Proof of Theorem \ref{T1}.} Choose for all $i\geq 0$ a constant $0<b_{i}<l_{i}/2$, so that $-\sum_{i\geq 0}b_{i}\log b_{i}<\infty$. Let $z=[x,y]\in \bigcap_{k\in\Z}\phi^{k}(M^{\ast})$. For all $n\geq 0$, let $k(n)$ be the positive integer such that $\phi^{n}(z)=[T^{k(n)}x,y']$ for some $y'\in (-r(T^{k(n)-1}x),r(T^{k(n)}x))$. Hence, we have
\begin{eqnarray*}
\|d_{z}\phi^{n}\|^{e}&\leq& 2+2\sum_{i=0}^{k(n)}|r'(T^{i}x)|\\
&\leq& \sum_{i=0}^{k(n)}(2+2|r'(T^{i}x)|)\leq \sum_{i=0}^{n-1}h(T^{i}x).
\end{eqnarray*}
From Proposition \ref{lemmaAaronson}, we obtain for $\mu$-almost every $z\in M$
\begin{eqnarray*}
\lim_{n\vers\infty}\frac{1}{n}\log^{+} \|d_{z}\phi^{n}\|^{e}&\leq& \lim_{n\vers\infty}\frac{1}{n}\log^{+}\left(\sum_{i=0}^{n-1}h(T^{i}x)\right)\\
&=& 0.
\end{eqnarray*}
In particular, the positive Lyapunov exponents for $\phi$, with respect to $g^{e}$, are $\mu$-a.e. equal to zero since for all $v\in T_{z}M$ we have
\begin{eqnarray*}
\lim_{n\vers\infty}\frac{1}{n}\log\|d_{z}\phi^{n}(v)\|&\leq& \lim_{n\vers\infty}\frac{1}{n}\log(\|d_{z}\phi^{n}\|^{e}\|v\|^{e})\\
&=&\lim_{n\vers\infty}\frac{1}{n}\log \|d_{z}\phi^{n}\|^{e} = 0.
\end{eqnarray*}
The same argument for $\phi^{-1}$ implies that the negative Lyapunov exponents for $\phi$, with respect to $g^{e}$, are $\mu$-a.e. equal to zero.

Let $\lambda^{\delta}(z,v)$ be the Lyapunov exponent for $z\in M$ in the direction of $v\in T_{z}M$ with respect to the metric $g^{\delta}$. If $z=[x,y]\in \bigcap_{k\in\Z}\phi^{k}(M^{\ast})$, we have
\begin{eqnarray*}
\lambda^{\delta}(z,v)&=&\lim_{n\vers\infty}\frac{1}{n}\log \|d_{z}\phi^{n}(v)\|^{\delta}_{\phi^{n}z}\\
&\leq& \lim_{n\vers\infty}\frac{1}{n}\log\|d_{z}\phi^{n}\|^{\delta}\|v\|^{\delta}_{z}\\
&=& \lim_{n\vers\infty}\frac{1}{n}\log\|d_{z}\phi^{n}\|^{\delta}\\
&\leq& \lim_{n\vers\infty}\frac{1}{n}\log(2+2|r'(x)|)\\
& &+ \lim_{n\vers\infty}\frac{1}{n}\log \max\{(2+2|r'(x)|),(2+2|r'(T^{k(n)}x)|)\}\\
& &+ \lim_{n\vers\infty}\frac{1}{n}\log \|d_{z}\phi^{n}\|^{e}\\
&=&  \lim_{n\vers\infty}\frac{1}{n}\log h(T^{k(n)}x).
\end{eqnarray*}
Since the function $\log^{+}h$ is $m$-integrable, Birkhoff Ergodic Theorem implies
\begin{eqnarray*}
0\leq \lim_{n\vers\infty}\frac{1}{n}\log h(T^{k(n)}x)&=& \lim_{n\vers\infty}\frac{k(n)}{n}\frac{1}{k(n)}\log h(T^{k(n)}x)\\
&\leq& \lim_{n\vers\infty}\frac{1}{k(n)}\log (h(T^{k(n)}x)\\
&=& 0,
\end{eqnarray*}
for $m$-almost every $x\in I$. Hence, the positive Lyapunov exponents for $\phi$, with respect to $g^{\delta}$ are $\mu$-a.e. are equal to zero. The same argument for $\phi^{-1}$ implies this fact for the negative Lyapunov exponents.\\

From Abramov's formula we have
$$h_{\mu}(\phi)=\frac{h}{2\int r dm}.$$
Using the fact that $h_{\mu}(\phi^{s})=|s|h_{\mu}(\phi)$ for all $s\in \R$, we have
$$h_{\mu}(f)=h,$$
for $f=\phi^{2\int r dm}$. Since the Lyapunov exponents for $f$ are equal to zero $\mu$-a.e., this concludes the proof of Theorem \ref{T1}.
\begin{flushright}
$\square$
\end{flushright}

\addcontentsline{toc}{section}{References}
\bibliography{biblio}
\bibliographystyle{amsalpha}


\end{document}